\documentclass{amsart}[12pt]

\usepackage{amsthm, amssymb, tikz}

\title{Double {H}urwitz numbers via the infinite wedge}

\author{Paul Johnson}
\address{Paul Johnson, Department of Mathematics, Imperial College London, 180 Queen's Gate, London SW7 2AZ, UK}
\email{paul.johnson@imperial.ac.uk}

\thanks{Supported in part by NSF grants DMS-0602191 and DMS-0902754.}

\begin{document}

\newcommand{\C}{\mathbb{C}}
\newcommand{\Z}{\mathbb{Z}}
\newcommand{\Q}{\mathbb{Q}}
\newcommand{\R}{\mathbb{R}}

\newcommand{\A}{\mathcal{A}}
\newcommand{\proj}{\mathbb{P}}
\newcommand{\Aut}{\textrm{Aut}}
\newcommand{\SSS}{\mathcal{S}}
\newcommand{\infwedge}{\bigwedge^{\frac{\infty}{2}}}
\newcommand{\E}{\mathcal{E}}
\newcommand{\EE}[2]{\E(#1,#2)}
\newcommand{\bvs}[4]{\varsigma\left(\begin{smallmatrix} #1 & #2 \\ #3 & #4 \end{smallmatrix}\right)}
\newcommand{\comment}[1]{}

\theoremstyle{definition}
\newtheorem{dummy}{}[section]
\newtheorem{theorem}[dummy]{Theorem}
\newtheorem{lemma}[dummy]{Lemma}
\newtheorem{example}[dummy]{Example}
\newtheorem{corollary}[dummy]{Corollary}
\newtheorem{definition}[dummy]{Definition}
\newtheorem{remark}[dummy]{Remark}
\newtheorem*{MNtheorem}{Theorem (Murnaghan-Nakayama)}
\newtheorem*{Frobtheorem}{Theorem (Frobenius)}
\newtheorem*{maintheorem}{Main Theorem (rough statement)}
\begin{abstract}
We derive an algorithm to produce explicit formulas for certain generating functions of double Hurwitz numbers.  These formulas generalize a formula in \cite{GJV} for one part double Hurwitz numbers.  Immediate consequences include a new proof that double Hurwitz numbers are piecewise polynomial, an understanding of the chamber structure and wall crossing for these polynomials, and a proof of the Strong Piecewise Polynomiality conjecture of \cite{GJV}.

The method is a straightforward application of Okounkov's expression \cite{OHur} for double Hurwitz numbers in terms of operators on the infinite wedge.  We begin with a introduction to the infinite wedge tailored to our use.
\end{abstract}

\maketitle

\setcounter{tocdepth}{1}
\tableofcontents

\section{Introduction}

\subsubsection{Double Hurwitz Numbers $H_g(\mu,\nu)$}

Hurwitz numbers count the number of maps between Riemann surfaces with specified ramification.

For $\mu$ a partition, we will use $\ell(\mu)$ to denote the number of parts of $\mu$, and $|\mu|$ to denote the {\em size} of $\mu$; that is, $\mu_1+\cdots+\mu_{\ell(\mu)}=|\mu|$.  We will always use $\ell(\mu)=m$ and $\ell(\nu)=n$.

The {\em Double Hurwitz Number} $H_g(\mu,\nu)$ counts the number of maps $f:\Sigma\to\proj^1$, where $\Sigma$ is a connected complex curve of genus $g$, and $f$ has profile $\mu$ over $0$, $\nu$ over infinity, and simple ramification over $r=2g-2+\ell(\mu)+\ell(\nu)$ other fixed points.  Each map $f$ is counted with weight $\frac{1}{|\Aut(f)|}$, where $\Aut(f)$ denotes the subgroup of automorphisms of $\Sigma$ that commute with $f$.  We further require that the automorphisms fix $f^{-1}(0)$ and $f^{-1}(\infty)$ pointwise.  This extra condition is taken in \cite{GJV}, and has the result of multiplying Hurwitz numbers by $|\Aut(\mu)|\cdot|\Aut(\nu)|$.

As $r$ is frequently more natural than $g$, we will use the notation $H^r(\mu,\nu)=H_g(\mu,\nu)$.
Furthermore, we organize the double Hurwitz numbers with fixed $\mu$ and $\nu$ and varying genus into the following series in $z$, which we will call the \emph{ $m+n$ point series}:
\begin{equation*}
H_{\mu,\nu}(z)=\sum_{g=0}^{\infty} \frac{z^r}{r!}H^r(\mu,\nu).
\end{equation*}

\subsubsection{}
Our main result, Theorem \ref{thm-main}, is an algorithm for computing $H_{\mu,\nu}(z)$. For any given $\mu$ and $\nu$, the algorithm produces an attractive closed form expression for the series.  The algorithm is an easy consequence of an expression of Okounkov for double Hurwitz numbers in terms of the infinite wedge \cite{OHur}, and further development of this point of view by Okounkov and Pandharipande \cite{OP1, OP2}.

The motivation of this paper is not so much Theorem \ref{thm-main} itself, but some of its consequences.  It allows us to easily reprove and extend most of the known results about the algebraic structure of $H(\mu,\nu)$.  As corollaries of Theorem \ref{thm-main}, we give a new proof that $H(\mu,\nu)$ is piecewise polynomial, and extend this to a full solution of their Strong Piecewise Polynomial conjecture.  Furthermore, we derive a novel wall crossing formula, similar to but simpler than that found in \cite{CJM2}.

\subsubsection{}
The structure of the paper is as follows.  In the remainder of the introduction Section \ref{sec-statementofresults} carefully states our results, and Section \ref{sec-statementofmethods} gives an overview of the methods used.

Section \ref{sec-infwedge} is an introduction to the infinite wedge. The sole aim is accessibility for the reader unfamiliar with the infinite wedge. We make no pretense of completeness or concision and present no careful proofs.  Instead we focus rather narrowly on giving some intuition for the infinite wedge and its use as a tool for computing characters of the symmetric group.   This section contains no new material, and is essentially an extended exposition of a few formulas from \cite{OHur, OP1, OP2}.

Section \ref{sec-mainformula} contains the proof of the main result and several examples. The proof is an application of standard techniques for computing vacuum expectations on the infinite wedge.

Finally, Section \ref{sec-consequences} presents the consequences of the main formula to the structure of double Hurwitz numbers, including the proof of the Strong Piecewise Polynomiality Conjecture, a wall crossing formula, and a few additional observations.

\subsection{Statement of Results} \label{sec-statementofresults}

The series $H_{\mu,\nu}(z)$ has a chamber dependence on the value of $\mu,\nu$ which we now describe.

\subsubsection{Connected vs. disconnected covers and the resonance arrangement}

It is frequently more natural to relax the condition that $\Sigma$ be connected, and count disconnected covers.  For double Hurwitz numbers, these two counts actually agree for most values of $\mu$ and $\nu$ -- generically, a double Hurwitz cover is connected, as we now argue.

Fix a component of a disconnected cover.  This component must contain some subset $I\subset [m]$ and $J\subset [n]$ of the marked points, and map to $\proj^1$ with some degree $d^\prime<d$.  Thus, we must have $\sum_{i\in I}\mu_I =d^\prime=\sum_{j\in J}\nu_J.$

This discussion naturally leads us to define
\begin{definition}
The \emph{resonance arrangement} is a set of hyperplanes $W_{I,J}$ intersecting the region
\begin{equation*}
R_{m,n}=\left\{(\mu_1,\dots, \mu_m,\nu_1,\dots,\nu_n)\in \R^{m+n} \Bigg|\; \mu_i,\nu_j>0;\; \sum_{i=1}^n\mu_i=\sum_{j=1}^n \nu_j \right\}.
\end{equation*}

For $I\subset [m], J\subset [n]$ proper subsets, the hyperplane, or wall, $W_{I,J}$ is the set of $\mu,\nu\in R_{m,n}$ satisfying
\begin{equation*}
\sum_{i\in I}\mu_I =\sum_{j\in J}\nu_J.
\end{equation*}

A chamber $\mathfrak{c}$ of the resonance arrangement is a connected component of the complement of the $W_{I,J}$ inside $R_{m,n}$.
\end{definition}

The walls $W_{I,J}$ are precisely the values of $\mu,\nu$ which allow for disconnected covers.  If we fix a chamber $\mathfrak{c}$ and look at $\mu,\nu\in\mathfrak{c}$, connected and disconnected Hurwitz numbers agree; unless stated otherwise we will assume we are in this situation.

\subsubsection{The main theorem}

Following \cite{OP1}, we introduce the following notation for functions closely related to the hyperbolic sine:
$$\varsigma(z)=e^{z/2}-e^{-z/2}=2\cdot\textrm{sinh}(z/2)$$
and
$$\SSS(z)=\frac{\varsigma(z)}{z}=\frac{\textrm{sinh}(z/2)}{z/2}.$$

Double Hurwitz numbers with $m=1$; that is, $\mu=(d)$, are called one part double Hurwitz numbers.  It is easy to see that $R_{1,n}$ consists of only one chamber.  In \cite{GJV}, the following formula for one part double Hurwitz numbers, i.e. double Hurwitz numbers with $\mu=(d)$, was derived:
\begin{equation} \label{GJVasStated}
H_g(d,\nu)=r!d^{r-1}[t^{2g}]\frac{1}{\SSS(t)}\prod_{i=1}^n \SSS(\nu_it).
\end{equation}
Here, $[t^{2g}]$ means take the coefficient of $t^{2g}$ in the expression that follows.
Equation (\ref{GJVasStated}) is easily seen to be equivalent to the following formula for the $1+n$-point series:
\begin{equation} \label{GJVFormula}
H_{(d),\nu}(z)=\frac{1}{d}\frac{1}{\prod \nu_i}\frac{1}{\varsigma(dz)} \prod_{i=1}^n \varsigma(d\nu_iz).
\end{equation}

Our main result, Theorem \ref{thm-main}, is the natural generalization of Equation (\ref{GJVFormula}) to all $m+n$ point series.  Though slightly difficult to state, it expresses the general $m+n$ point series as a sum of terms with the same general form as those appearing in Equation (\ref{GJVFormula}). More specifically, it says:
\begin{maintheorem}
For $\mu,\nu\in \mathfrak{c}$, a chamber of $R_{m,n}$, we have
\begin{equation*}
H_{\mu,\nu}(z)=\frac{1}{\prod \mu_i}\frac{1}{\prod \nu_j} \frac{1}{\varsigma(dz)} \sum_{k=1}^{t(\mathfrak{c})} \prod_{\ell=1}^{m+n-1} \varsigma(zQ^{\mathfrak{c}}_{k,\ell}),
\end{equation*}
where $t(\mathfrak{c})$ is finite and the $Q^{\mathfrak{c}}_{k,\ell}$ are certain quadratic polynomials in $\mu_i$ and $\nu_j$.
\end{maintheorem}

Theorem \ref{thm-main} is presented in Section \ref{sec-mainformula}.  Though the double Hurwitz number $H^r(\mu,\nu)$ does not depend on an ordering of the parts of $\mu$ and $\nu$, the precise form of the sum in Theorem \ref{thm-main} does. Different orderings result in different expressions that are equivalent due to identities for $\varsigma$ and the fact that $|\mu|=d=|\nu|$.  This is concretely illustrated in Examples \ref{example-221} and \ref{example-222} in Section \ref{sec-ordering}.

Apart from providing a convenient way to calculate Hurwitz numbers, Theorem \ref{thm-main} has several immediate consequences about the structure of double Hurwitz numbers.  These are presented in full in Section \ref{sec-consequences}; we state the results now.

\subsubsection{Polynomiality}

Our main motivation was the following corollary:
\begin{corollary} \label{maincorollaries}
For $\mu,\nu$ restricted to a given chamber $\mathfrak{c}$ of the resonance arrangement, we have
\begin{equation*}
H_g(\mu,\nu)=P^{\mathfrak{c},r}(\mu,\nu)=\sum_{k=0}^g (-1)^{k}P^{\mathfrak{c}}_{g,k}(\mu,\nu),
\end{equation*}
where $P^{\mathfrak{c}}_{g,k}$ is a homogeneous polynomial of degree $4g-3+m+n-2k$, with $P^\mathfrak{c}_{g,k}(\mu,\nu)>0$ for $(\mu,\nu)\in\mathfrak{c}$.
\end{corollary}

Corollary \ref{maincorollaries} is essentially the \emph{Strong Piecewise Polynomial conjecture} of \cite{GJV}.  We now survey previous results in this direction.   It was proven in \cite{GJV} that $H_g(\mu,\nu)$ was piecewise polynomial of degree $4g-3+m+n$, and Corollary \ref{maincorollaries} was proven in full for one part double Hurwitz numbers.   The walls of polynomiality were first determined in genus zero by \cite{SSV}, reproduced in genus zero in \cite{CJM1}, and extended to all genera in \cite{CJM2}, which also proves that $H_g(\mu,\nu)$ is either even or odd.  To our knowledge, this is the first proof of the lower degree bound and of the positivity properties of the polynomials.

The motivation and form of the Strong Piecewise Polynomial Conjecture is another, deeper conjecture in \cite{GJV}, which we will call the GJV conjecture, that would express double Hurwitz numbers as intersection numbers of certain tautological cycles in a (as yet undetermined) compactification of the universal Picard variety, a moduli space parameterizing curves together with a holomorphic line bundle.   The GJV conjecture would give a geometric explanation of all aspects of Corollary \ref{maincorollaries}, parallel to the way that the ELSV formula \cite{ELSV, GV, L} explains similar facts about polynomiality for single Hurwitz numbers.

   As our methods are entirely algebraic, they make no progress toward proving the GJV conjecture.  However, corollary \ref{maincorollaries} can be interpreted as providing evidence for an extension of the GJV conjecture.  In \cite{GJV}, the conjecture is only stated for one-part double Hurwitz numbers.  Corollary \ref{maincorollaries} essentially says that the algebraic structure of double Hurwitz numbers on other chambers support an extension of the conjecture. See \cite{CJM2} for further evidence in support of this extension.  In \cite{GJV},  ad-hoc definitions of the compactified Picard variety in genus 0 and 1 gave additional support to the conjecture.  Similar checks on other chambers is clearly desirable.

As an additional consequence of our main theorem, we see that $P^{\mathfrak{c}}_{g,k}$ is essentially $P^{\mathfrak{c}}_{g-k,0}$:
\begin{corollary} \label{bernoullicorollary}
\begin{equation*}
P^{\mathfrak{c}}_{g,k}(\mu,\nu)=\frac{(1-\frac{1}{2}^{2k-1})|B_{2k}|d^{2k}}{(2k)!} P^{\mathfrak{c}}_{g-k,0}(\mu,\nu),
\end{equation*}
\end{corollary}
where $B_{2k}$ denote the Bernoulli numbers.

\subsubsection{Wall Crossing}
From our main theorem we also derive a wall crossing formula for $H_{\mu,\nu}(z)$, and hence $H_g(\mu,\nu)$.

By a wall crossing formula for $H_{\mu,\nu}(z)$, we mean the following.  Let $\mathfrak{c}_1$ and $\mathfrak{c}_2$ be two adjacent chambers of the resonance arrangement, adjacent along the wall $W_{I,J}$.  Then Theorem \ref{thm-main} gives two series expansions $S^1_{\mu,\nu}(z)$ and $S^2_{\mu,\nu}(z)$ that agree with $H_{\mu,\nu}(z)$ for $(\mu, \nu)$ integers in chamber $\mathfrak{c}_1, \mathfrak{c}_2$, respectively.  However, these series $S^i_{\mu,\nu}(z)$ make sense for arbitrary $(\mu,\nu)$ and so their difference makes sense as well.

\begin{definition}
\begin{equation*}
WC^{I,J}_{\mu,\nu}(z)=S^2_{\mu,\nu}(z)-S^1_{\mu,\nu}(z).
\end{equation*}
\end{definition}

Note that since $|\mu|=|\nu|$, the wall $W_{I,J}$ is equivalent to the wall $W_{I^c, J^c}$.  Following \cite{CJM2}, we will use this redundancy to indicate which direction we are crossing the wall: we will always move so that $\delta=|\mu_I|-|\nu_J|=|\nu_J^c|-|\mu_I^c|$ is increasing; that is, it will be negative on $\mathfrak{c}_1$, zero on the wall $W_{I,J}$, and positive on $\mathfrak{c}_2$.

We will use $\delta=|\mu_I|-|\nu_J|, d_1=|\mu_I|=|\nu_J|+\delta$ and $d_2=|\mu_I^c|+\delta=|\nu_J^c|$.

\begin{theorem} \label{thm-wallcrossing}
For $\mu,\nu\in\mathfrak{c}_1$, we have
\begin{equation*}
WC^{I,J}_{\mu,\nu}(z)=\delta^2\frac{\varsigma(d_1z)}{\varsigma(\delta d_1z)}
\frac{\varsigma(d_2z)}{\varsigma(\delta d_2z)}
\frac{\varsigma(\delta dz)}{\varsigma(dz)}
H_{\mu_I,\nu_J+\delta}(z)H_{\mu_I^c+\delta, \nu_J^c}(z).
\end{equation*}
\end{theorem}

In $\mathfrak{c}_2, \delta>0$, and we will see that the chamber $\mathfrak{c}_2$ determines chambers for the smaller Hurwitz numbers $\mu_I, \nu_J+\delta$ and $\mu_I^c+\delta, \nu_J^c$.

This is a natural generalization of the wall crossing formula in \cite{SSV} given for genus 0 Hurwitz numbers.  To extract the genus formula, we take the asymptotics as $z\to 0$.  The left hand side becomes the genus zero wall crossing, and on the right hand side all the $\varsigma$ terms together simplify to $1/\delta$.

Theorem \ref{thm-wallcrossing} has the same general form to that given in \cite{CJM2}.  It is in some sense much simpler in that we only need to add a single part of size $\delta$ to the smaller Hurwitz numbers instead of summing over all partitions of $\delta$, and the product of $\varsigma$'s involved is also simpler than the inclusion/exclusion in \cite{CJM2}.  However, the formula in \cite{CJM2} has the advantage of being entirely in terms of Hurwitz theory.

\comment{In addition to the walls $W_{I,J}$, it was noted in \cite{CJM2} that when a part $\mu_i$ goes to zero, it is natural to think of it as a positive part of $\nu$, and so have a wall crossing between $(m,n)$ Hurwitz numbers and $(m-1, n+1)$ Hurwitz numbers, and that this wall crossing is as well behaved as the interior wall crossing.  However, there was no interpretation of what the Hurwitz polynomial should mean when we set a variable equal to zero.  We derive the following natural interpretation in terms of Hurwitz numbers with one less part:
\begin{theorem} \label{thm-zero}
Suppose that $\mathfrak{c}$ is a chamber of $R_{m,n}$ bordering $\mu_m=0$. \begin{equation*} 
P^{\mathfrak{c},r}(\mu_1,\dots, \mu_{m-1}, 0, \nu_1,\dots, \nu_n)=rd P^{\mathfrak{c}, r-1}(\mu_1,\dots,\mu_{m-1}, \nu_1,\dots,\nu_n).
\end{equation*}
\end{theorem} }

\subsubsection{Chambers with product formulas}

Although in general our main formula Theorem \ref{thm-main} only expresses $H_{\mu,\nu}(z)$ as a sum of terms, in certain chambers it has an expression with only one term.  In Section \ref{sec-specialchambers} we determine all chambers where Theorem \ref{thm-main} has only one term, giving a product formula for double Hurwitz numbers.  We note that it is possible that double Hurwitz numbers can be written this way on other chambers -- we only show that our algorithm cannot do this.  Furthermore, it turns out that in these chambers the polynomials $Q_{k, \ell}(\mu,\nu)$ can explicitly be determined, giving a completely closed form expression for $H_{\mu,\nu}(z)$ on these chambers.

A subset of these chambers include the \emph{totally negative} chambers of \cite{SSV}, where they gave explicit formulas for genus zero double Hurwitz numbers.  Our formula reproduces theirs and extends it to higher genus, but contains additional chambers as well.

\subsection{Technique of Proof} \label{sec-statementofmethods}
Although we have defined Hurwitz numbers as a geometric object, our techniques are entirely algebraic.   It is a classical result, responsible for much of the interest in Hurwitz numbers, that studying the monodromy of the cover reduces Hurwitz numbers to counting certain sets of elements in the symmetric group (or more general monodromy groups).

It is nearly as old of a result (Okounkov \cite{OHur} points out that it is an exercise in Burnside's group theory textbook), that this group theoretic count is conveniently calculated using representation theory.  This is the method used to prove Equation (\ref{GJVasStated}) in \cite{GJV}, and it is the method we will follow as well.

We will not discuss the transition between counting ramified covers and the character theory of the symmetric group further.  Detailed discussions can be found in \cite{graphsonsurfaces} or \cite{roth}.

\subsubsection{Hurwitz numbers in terms of character theory of $S_n$}
We will now state the formula for double Hurwitz numbers in terms of representation theory, which will be the starting point of our algorithm.

Both representations and conjugacy classes of the symmetric group are naturally indexed by partitions.  For us, $\lambda$ will always be a partition indexing a representation, while $\mu$ and $\nu$ will always index conjugacy classes.

Let $\chi^\lambda_\mu$ denote the character indexed by $\lambda$, $\dim \lambda$ denote the dimension of the representation indexed by $\lambda$, and let $C_\mu$ denote the size of the conjugacy class denoted by $\mu$.  Then the central character $f_\mu(\lambda)$ is
\begin{equation*}
f_\mu(\lambda)=\frac{C_\mu}{\dim\lambda} \chi^\lambda_\mu.
\end{equation*}

We will use $f_2(\lambda)$ to denote the central character of a transposition; that is, when $\mu=2+1+\cdots+1$.

Then \emph{disconnected} double Hurwitz numbers can be expressed as:
\begin{equation} \label{eq-HurwitzRep}
H^{r}(\mu,\nu)=\frac{1}{\prod \mu_i}\frac{1}{\prod \nu_j} \sum_{|\lambda|=d} \chi^\lambda_\mu f_2(\lambda)^r \chi^\lambda_\nu.
\end{equation}

As the degree $d$ grows, this expression appears to become quite complicated:  the number of partitions of $d$ grows exponentially, and calculating a general character value can become complex.  In some sense our main result is a consequence of these complications not arising.   We now explain this heuristically, beginning with the one part case.

\subsubsection{The one part case}
To see that Equation (\ref{eq-HurwitzRep}) is not that complicated in the one part case, first note that the character $\chi^\lambda_{(d)}$ vanishes for most representations $\lambda$ -- in fact, it vanishes unless $\lambda$ is an L-shaped partition $\lambda=k+1+1+1+\cdots 1$.  Thus the sum over partitions in Equation (\ref{eq-HurwitzRep}) really only receives contributions from $d$ different partitions.

The situation further simplifies because the representations indexed by the L-shaped partitions are quite simple: they are exactly the exterior powers of the standard $(d-1)$ dimensional representation. From this fact it follows that there are simple and explicit formulas for their characters that can be derived and packaged into generating functions ``by hand'', and the generating function viewpoint leads quickly to Equation (\ref{GJVasStated}).

\subsubsection{The general case}
Though more complicated than the one part case for the $m+n$ point series is still much simpler than it might initially appear.  Again, the sum over partitions in \ref{eq-HurwitzRep} will vanish on most partitions: by the Murnaghan-Nakayama rule, $\chi^\lambda_\mu$ will vanish unless $\lambda$ is composed of at most $m$ border strips, and clearly the number of such partitions grows polynomially in $d$.

Furthermore, since the number of parts of $\mu,\nu$ is fixed, the Murnaghan-Nakayama rule is relatively efficient for calculating $\chi^\lambda_{\mu}$.  A formula of Frobenius can calculate $f_2(\lambda)$ explicitly, and so computationally we have seen that Equation (\ref{eq-HurwitzRep}) is well-behaved in the general case.

So, the only difficulty in extending the argument of \cite{GJV} to the case of the general $m+n$ point series is packaging this computation attractively in terms of generating functions.  Presumably, with ingenuity this could again be done ``by hand''. Instead, we will accomplish it by writing Equation \ref{eq-HurwitzRep} in terms of operators acting on the infinite wedge.  Then, commutator relationships for these operators will provide an algorithmic way to derive the desired formulas.

\subsubsection{The infinite wedge and Hurwitz theory}

An introduction to the mathematics of the infinite wedge is contained in \ref{sec-infwedge} -- here, we only comment that applying it to Hurwitz theory is hardly original to this work.

The use of the infinite wedge in Hurwitz theory began in the physics literature. An article accessible to mathematicians, with pointers to more physics literature, is Dijkgraaf \cite{D}.   Dijkgraaf uses a more general form of Equation \ref{eq-HurwitzRep} is used to express Hurwitz numbers counting covers of a torus as the trace of an operator on the infinite wedge, from which it is deduced that they are quasimodular forms.

In the same vein, Okounkov \cite{OHur} expressed double Hurwitz numbers as vacuum expectations of a certain operator on the infinity wedge to show that they satisfy the 2-d Toda hierarchy.  Okounkov and Pandharipande developed this machinery further in their calculation of the Gromov-Witten theory of curves \cite{OP1, OP2}, which heavily utilizes Hurwitz numbers.

Our main algorithm is a straightforward application of the results of \cite{OHur, OP1, OP2}.  The main novelty is applying this technique to the question of polynomiality, and the expanded exposition.

\subsection{Acknowledgements}

This work would not have been possible without the help of my advisor, Yongbin Ruan: I began learning the machinery used here in the very first reading course I took with him, and learned the rest while writing my thesis.  I also thank Renzo Cavalieri and Hannah Markwig for conversations during the related work \cite{CJM1, CJM2}, from which this paper developed.

\section{The infinite wedge} \label{sec-infwedge}

The infinite wedge has connections and applications to a vast range of mathematics -- representation theory, integrable systems, and modular forms, to name a few.  A full introduction to the infinite wedge is beyond the scope of this note.   However, we will only be using a few aspects of the infinite wedge as a tool for dealing with the character theory of the symmetric group.  This section is a self-contained introduction to this aspect infinite wedge. We sometimes present more material than we will need when doing so will help aid intuition.  More detailed presentations of the infinite wedge include \cite{Owedge}, \cite{KR}, and \cite{MJD}.

\subsubsection{}
A frequent approach to the character theory of the symmetric group is through symmetric functions.  The ring of symmetric functions has many different bases, in particular the power sum functions $p_\mu$ and the Schur functions $s_\lambda$.  The change of basis matrix between $p_\mu$ and $s_\lambda$ is essentially the character table of $\chi^\lambda_\mu$ of $S_n$.

For us, the infinite wedge $\infwedge V$ will play the same role as the ring of symmetric functions.  In fact, there is a natural isomorphism between (the charge zero part of) the infinite wedge and the ring of symmetric functions, known as the Boson-Fermion correspondence, and so in principle all calculations we do could be carried in terms of symmetric functions.

An easy way to understand the benefit of working with the infinite wedge rather than symmetric functions is the following: the Schur functions $s_\lambda$ are frequently held up as the ``best'' basis of the ring of symmetric functions.  In terms of representation theory, this is seen in that the Schur functions are the basis corresponding to the characters, which are the semi-simple basis of the center of the group algebra.  However, Schur functions are the most complicated of the usual symmetric functions to define.

In the infinite wedge, on the other hand, it is the most natural basis $v_\lambda$ that corresponds to the characters (and hence the Schur functions); it is the basis corresponding to conjugacy classes that is more complicated.

\subsection{Definition of $\infwedge V$}

Let $V$ be the vector space with basis labeled by the half-integers.  We use the underscore to represent the corresponding basis vector, so that $\underline{1/2}$ is the basis vector indexed by $1/2$, and so
\begin{equation*}
V=\bigoplus_{i\in\Z} \underline{i+\frac{1}{2}}.
\end{equation*}

\begin{definition} \label{def-infwedge}
The infinite wedge $\infwedge V$ is the span of vectors of the form
\begin{equation*}
\infwedge V=\bigoplus_{ (i_k)} \underline{i_1}\wedge \underline{i_2}\wedge\cdots
\end{equation*}
where the sum is over all decreasing sequences of half integers $i_k\in\Z+\frac{1}{2}$ such that
\begin{equation} \label{infcondition}
i_k+k-1/2=c \textrm{ for $k$ sufficiently large}.
\end{equation}
Here, $c$ is some constant known as the \emph{charge}.  This terminology is borrowed from physics, and will be explain in Sections \ref{sec-dirac} and \ref{sec-CnE}
\end{definition}

\subsubsection{}
Since the sequence $(i_k)$ is decreasing, we can recover it from the set of values $S=\{i_k\}$.  Let $\Z^+_{1/2}$ and $\Z^-_{1/2}$ be the positive and negative elements of $\Z+1/2$, respectively.  Then the fact that $(i_k)$ is decreasing implies that
\begin{equation} \label{con1}
S\cap \Z^+_{1/2} \text{ is finite}
\end{equation}
while condition \ref{infcondition}
implies that
\begin{equation} \label{con2}
S^c \cap \Z^-_{1/2} \textrm{ is finite}.
\end{equation}
Furthermore, from any $S$ satisfying conditions (\ref{con1}) and (\ref{con2}) we can construct a decreasing sequence $i_k$ satisfying condition (\ref{infcondition}). Using this correspondence, we will use $v_S$ to denote the vector $\underline{i_1}\wedge \underline{i_2}\wedge \cdots$.

The basis vectors $v_S$ are conveniently pictured graphically by a collection of black and white stones known as {\em Maya Diagrams}: a stone is placed at each half integer on the number line (with, conventionally, the negative direction going to the right).  For each $k\in\Z+1/2$, if $k\in S$ we place a black stone, while if $k\notin S$ we place a white stone.  Sometimes we will treat the white stones as an empty space.  Conditions (\ref{con1}) and (\ref{con2}) are equivalent to the fact that sufficiently far to the left all the stones will be white, while far to the right all the stones will be black.

\subsubsection{}
We will be interested only in the charge zero subspace of the infinite wedge $\infwedge_0 V$, where the charge $c=0$; that is, when $i_k+k-1/2=0$ for large $k$.  Letting $\lambda_k=i_k+k-1/2$, we see that $\lambda_k$ is a decreasing and eventually zero, and hence a partition. Thus, $\infwedge_0 V$ has a basis $v_\lambda$ indexed by partitions:
\begin{equation} \label{algebraicvlambda}
v_\lambda=\underline{\lambda_1-1/2}\wedge \underline{\lambda_2-3/2}\wedge \underline{\lambda_3-5/2}\wedge\cdots\wedge \underline{\lambda_i-i+1/2}\cdots
\end{equation}

The algebraic translation between $\lambda$ and the Maya diagram given in Equation (\ref{algebraicvlambda}) is interpreted graphically in Figure \ref{graphicvlambda}, which we now describe.  Draw the partition $\lambda$ in Russian notation, that is, rotated counter-clockwise by 45 degrees and enlarged by a factor of $\sqrt{2}$.  Place the Maya diagram beneath the partition, with $0$ beneath the vertex; then one stone will lie beneath each edge of $\lambda$.  Downward sloping edges of $\lambda$ correspond to white stones of the Maya diagram, while upward sloping edges of $\lambda$ correspond to black stones.

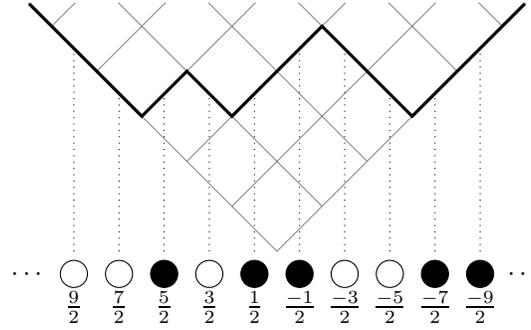
\begin{figure} \label{graphicvlambda}
\caption{Dictionary between partitions and Maya diagrams for $\lambda=3+2+2$}

\begin{tikzpicture}
\begin{scope}[gray, very thin, scale=.6]
\clip (-5.5, 5.5) rectangle (5.5, -5);
\draw[rotate=45, scale=1.412] (0,0) grid (6,6);
\end{scope}

\begin{scope}[rotate=45, very thick, scale=.6*1.412]
\draw (0,5.5) -- (0, 3) -- (1,3) -- (1,2) -- (3,2) -- (3,0) -- (5.5,0);
\end{scope}

\begin{scope}[scale=.6, dotted]
\draw (-4.5,0) -- (-4.5, 4.5);
\draw (-3.5,0) -- (-3.5, 3.5);
\draw (-2.5,0) -- (-2.5, 3.5);
\draw (-1.5,0) -- (-1.5, 3.5);
\draw (-.5,0) -- (-.5, 3.5);
\draw (.5,0) -- (.5, 4.5);
\draw (1.5,0) -- (1.5, 4.5);
\draw (2.5,0) -- (2.5, 3.5);
\draw (3.5,0) -- (3.5, 3.5);
\draw (4.5,0) -- (4.5, 4.5);
\end{scope}

\begin{scope}[scale=.6, yshift=-.5cm]
\draw (-5.5,0) node{$\cdots$};
\draw (-4.5,0) circle (.3) node[below=3pt]{$\frac{9}{2}$};
\draw (-3.5,0) circle (.3) node[below=3pt]{$\frac{7}{2}$};
\filldraw (-2.5,0) circle (.3) node[below=3pt]{$\frac{5}{2}$};
\draw (-1.5,0) circle (.3) node[below=3pt]{$\frac{3}{2}$};
\filldraw (-.5,0) circle (.3) node[below=3pt]{$\frac{1}{2}$};
\filldraw (.5,0) circle (.3) node[below=3pt]{$\frac{-1}{2}$};
\draw (1.5,0) circle (.3) node[below=3pt]{$\frac{-3}{2}$};
\draw (2.5,0) circle (.3) node[below=3pt]{$\frac{-5}{2}$};
\filldraw (3.5,0) circle (.3) node[below=3pt]{$\frac{-7}{2}$};
\filldraw (4.5,0) circle (.3) node[below=3pt]{$\frac{-9}{2}$};
\draw (5.5,0) node{$\cdots$};
\end{scope}
\end{tikzpicture}
\end{figure}

\subsubsection{}

We will largely be interested in certain operators acting on the infinite wedge.  These operators will be analogs of the following situation from the finite dimensional situation:  If $W$ is a finite dimensional representation of a lie algebra $\mathfrak{g}$, then the wedge $\wedge^k W$ is, too, where $\mathfrak{g}$ acts on $\wedge^k W$ by the Leibniz rule
$$g\cdot \left(w_1\wedge w_2 \wedge \cdots \wedge w_k\right)=\sum_{i=1}^k w_1\wedge \cdots \wedge (g\cdot w_i) \wedge \cdots \wedge w_k.$$

In particular, if $W$ has a basis $e_1,\dots, e_n$, then we can view $n\times n$ matrices as the Lie algebra $\mathfrak{gl}(W)$, and hence they act on $\wedge^k V$.

We will extend this situation by describing and action of certain $\infty \times \infty$ matrices on $\infwedge V$.  However, some care is needed, in defining what lie algebra we are using, and that everything is well defined.

Let $\mathfrak{gl}_\infty$ be the set of matrices with only finitely many nonzero entries; that is
$$\mathfrak{gl}_\infty=\bigoplus_{i,j\in\Z+1/2} E_{ij}.$$
Then it is then clear that the usual commutator makes $\mathfrak{gl}_\infty$ into a Lie algebra that acts on $V$, and the usual Leibniz rule makes $\infwedge V$ into a representation.  However, $\mathfrak{gl}_\infty$ is not large enough for our purposes.

Instead, we will use the algebra $\mathcal{A}_\infty$, consisting of those matrices with only finitely many nonzero {\em diagonals}, that is:
\begin{equation}
\mathcal{A}_\infty=\left\{\sum_{i,j\in\Z+1/2}a_{ij} E_{ij}\Bigg| a_{ij}=0\; \textrm{for } |j-i|>>0\right\}.
\end{equation}

It is then again easy to see that multiplication, and hence commutation, of elements of $\mathcal{A}_\infty$ are well defined and again in $\mathcal{A}_\infty$.  Defining the action of $\mathcal{A}_\infty$ on $\infwedge V$ is more complicated, and we will treat it while introducing examples in the next two sections.    The action of elements $E_{kk}$ on the main diagonal will be treated in Section \ref{sec-diagonal}, while the next section introduces some non-diagonal elements.

\comment{-- while the naive action on the infinite wedge still works for non-diagonal matrices, it no longer works for diagonal ones.  For example, the matrix corresponding to the identity (which we will denote $C=\sum E_{kk}$) clearly would produce an infinite sum when applied to any element of $\infwedge V$.  In the next section we will define and explain the use of some non-diagonal operators $\alpha_k$; we will postpone dealing with action of diagonal entries until Section \ref{sec-diagonal}.
}

\subsubsection{}
We give $\infwedge V$ an inner product $(\cdot , \cdot )$ by declaring the $v_S$ to be orthonormal.

A special role will be played by the \emph{vacuum vector} $|0\rangle\in\infwedge_0 V$, which is the vector corresponding to, equivalently: the partition $\lambda$ of zero; the set $S=\Z^-_{1/2}$; the Maya diagram where all stones left of zero are white and all stones right of zero are black.

For an operator $M$, we define the \emph{vacuum expectation of $M$}, $\left\langle M \right\rangle$, by:
\begin{equation}
\left\langle M \right\rangle=\left(\left|0\right\rangle,M \left|0\right\rangle\right)
\end{equation}

The goal of the remainder of this chapter is to explain how to express the generating function $H_{\mu,\nu}(z)$ as a vacuum expectation.

\subsection{The operators $\alpha_k$ and the Murnaghan-Nakayama rule} \label{sec-nondiagonal}

In this section we will treat the following operators in $\A_\infty$, for $n\neq 0$:
\begin{definition}
$$\alpha_n=\sum_{k\in\Z+1/2} E_{k, k+n}.$$
\end{definition}

\subsubsection{}
We begin by noting that the $\alpha_n$ have an elegant description in terms of the Maya diagrams.  The vector $\alpha_n v_S$ will be a signed sum of basis vectors $v_{S^\prime}$.  The terms of the sum are obtained by picking up a black stone in $S$, and trying to set it down $n$ steps to the right.  If there is already a black stone there, the result is zero.  If there is a white stone there, we pick it up it up and set it down where the black stone was, creating a new Maya diagram by switching the locations of these two stones.  This diagram corresponds to some set $S^\prime$, and the vector $v_{S^\prime}$ appears in $\alpha_n v_S$ with sign $(-1)^s$, where $s$ is the number of black stones the stone we moved jumped over.

It is not difficult to see that this action results in a finite sum.  We will have a nonzero term in $\alpha_n v_S$ whenever $v_S$ has a white stone $n$ steps to the right of a black stone.  Since all the stones sufficiently far to the right are black, and all stones sufficiently far to the left are white, this can only happen finitely many times.  This same argument shows that if $A_n$ is any operator of the form
\begin{equation*}
A_n=\sum_{k\in\Z+1/2} a_k E_{k, k+n}
\end{equation*}
with $n\neq 0$, then the naive action of $A_n$ on the infinite wedge is well defined.

\subsubsection{}

Although we have seen that the Leibniz rule gives a well defined action of the operators $\alpha_n$ on $\infwedge V$, this action is not a representation of the lie algebra $\mathfrak{gl}_\infty$.  As operators in $\mathcal{A}_\infty$, we would expect the $\alpha_n$ to commute.  However, as operators on $\infwedge V$ they satisfy the following commutation relation:
\begin{equation*}
[\alpha_n,\alpha_m]=n\delta_{n,-m}.
\end{equation*}

\subsubsection{}
The nice description of the action of $\alpha_n$ on Maya diagrams translates in a similarly attractive description of the action on the vectors $v_\lambda$.  Since the Maya diagrams of the basis vectors $v_{S^\prime}$ appearing in $\alpha_n v_S$ only differ from $S$ at two places (the two swapped stones), the partitions $\lambda^\prime$ appearing will be closely related to the partition $\lambda$.  The path taken by the border of $\lambda^\prime$ will deviate from that of $\lambda$ at the first place, then follow it in parallel for the $n$ steps the stone was moved, and the rejoin the path at the other modified place.  This results exactly in adding or removing (depending on the sign of $n$) a border strip of length $n$.  This is illustrated in a simple example in Figure (\ref{ex-alpha2}).

Furthermore, the number of black stones that were jumped, and hence the sign with which $v_{\lambda^\prime}$ appears, will be the number of upward steps taken in the border strip.

\begin{figure} \label{ex-alpha2}
\caption{Demonstration that $\alpha_2 v_{(3,2,2)}=v_{(3,2)}-v_{(3,1,1)}$ }
\begin{tikzpicture}
\begin{scope}[gray, very thin, scale=.6]
\clip (-5.5, 5.5) rectangle (5.5, -5);
\draw[rotate=45, scale=1.412] (0,0) grid (6,6);
\end{scope}

\begin{scope}[rotate=45, very thick, scale=.6*1.412]
\draw (0,5.5) -- (0, 3) -- (1,3) -- (1,2) -- (3,2) -- (3,0) -- (5.5,0);
\draw [red] (1,2) -- (1,1) -- (3,1);
\end{scope}

\begin{scope}[scale=.6, yshift=-.5cm, >=stealth]
\draw (-5.5,0) node{$\cdots$};
\draw (-4.5,0) circle (.3);
\draw (-3.5,0) circle (.3);
\filldraw (-2.5,0) circle (.3);
\draw (-1.5,0) circle (.3);
\filldraw [draw=red, fill=gray] (-.5,0) circle (.3);
\filldraw (.5,0) circle (.3) node[below=7pt, red]{-} ;
\filldraw [red] (1.5,0) circle (.3);
\draw (2.5,0) circle (.3);
\filldraw (3.5,0) circle (.3);
\filldraw (4.5,0) circle (.3);
\draw (5.5,0) node{$\cdots$};
\draw [->, very thick, red] (-.5,-.5) ..controls (-.5, -1.5) and (1.5, -1.5) .. (1.5,-.5);

\end{scope}

\begin{scope}[xshift=7.5cm]
\begin{scope}[gray, very thin, scale=.6]
\clip (-5.5, 5.5) rectangle (5.5, -5);
\draw[rotate=45, scale=1.412] (0,0) grid (6,6);
\end{scope}

\begin{scope}[rotate=45, very thick, scale=.6*1.412]
\draw (0,5.5) -- (0, 3) -- (1,3) -- (1,2) -- (3,2) -- (3,0) -- (5.5,0);
\draw [red] (2,2) -- (2,0) -- (3,0);
\end{scope}

\begin{scope}[scale=.6, yshift=-.5cm, >=stealth]
\draw (-5.5,0) node{$\cdots$};
\draw (-4.5,0) circle (.3);
\draw (-3.5,0) circle (.3);
\filldraw (-2.5,0) circle (.3);
\draw (-1.5,0) circle (.3);
\filldraw (-.5,0) circle (.3);
\filldraw [draw=red, fill=gray] (.5,0) circle (.3);
\draw (1.5,0) circle (.3) node[below=7pt, red]{+};
\filldraw [red] (2.5,0) circle (.3);
\filldraw (3.5,0) circle (.3);
\filldraw (4.5,0) circle (.3);
\draw (5.5,0) node{$\cdots$};
\draw [->, very thick, red] (.5,-.5) ..controls (.5, -1.5) and (2.5, -1.5) .. (2.5,-.5);

\end{scope}
\end{scope}
\end{tikzpicture}

\end{figure}

The description of the $\alpha_n$ operators in terms of partitions is exactly the recursion present in the Murnaghan-Nakayama rule:
\begin{MNtheorem}
\begin{equation*}
\chi^\lambda_\mu=\sum_{\substack{\lambda^\prime=\lambda-S \\ S \textrm{ is a border strip of length $\mu_1$}}} (-1)^{\textrm{sign($S$)}} \chi^{\lambda^\prime}_{\mu-\mu_1}.  \end{equation*}
Where Sign($S$) is the number of steps of the border step that run from the lower left to the upper right.
\end{MNtheorem}
So immediately we have
\begin{lemma} \label{lem-MNwedge}
For $\mu$ a partition of $d$,
\begin{equation*}
\prod_{i=1}^{\ell(\mu)} \alpha_{-\mu_i}\left|0\right\rangle=\sum_{|\lambda|=d}\chi^\lambda_\mu v_\lambda. \end{equation*}
Similarly, taking the adjoint, we have
\begin{equation*}
\prod_{i=1}^{\ell(\mu)} \alpha_{\mu_i} v_\lambda=\chi_\mu^\lambda |0\rangle.
\end{equation*}
\end{lemma}

From Lemma \ref{lem-MNwedge}, if we define
\begin{equation*}
\alpha_{-\mu}|0\rangle=\prod_{i=1}^{\ell(\mu)} \alpha_{-\mu}|0\rangle
\end{equation*}
then $\alpha_{-\mu}|0\rangle$ is a basis of $\infwedge_0 V$, and the change of basis between $v_\lambda$ and $\alpha_{-\mu}|0\rangle$ is exactly the character tables of the symmetric groups $S_n$.

\subsection{The operator $\mathcal{F}_2$ and an equation of Frobenius} \label{sec-diagonal}

\subsubsection{}
In the previous section, it was noted that for nondiagonal elements of $\mathcal{A}_\infty$, the usual Leibniz rule resulted in only finite sums.  This is not the case for diagonal elements, however.  For instance, the identity matrix is in $\mathcal{A}_\infty$, and when applied to the vacuum this results in an infinite sum.

Instead of the usual Leibniz action of the diagonal elements $E_{kk}$ on the infinite wedge, we will use the following action.
\begin{definition} \label{def-Ekk}
\begin{equation*}
E_{k,k}\cdot v_S=\left\{\begin{array}{ll}
v_S & k>0, k\in S \\ -v_S & k<0, k\notin S \\ 0 & \text{else} \end{array}\right. .
\end{equation*}
\end{definition}

First, some motivation and intuition for this definition.  Definition \ref{def-Ekk} is essentially comparing the naive $E_{kk}$ action on $v_S$ with the naive $E_{kk}$ action on the vacuum.

More precisely, let $E_{kk}^\prime$ denote the naive action on the infinite wedge; that is
\begin{equation*}
E^\prime_{kk}\cdot v_S=\delta_{k\in S} v_S .
\end{equation*}
Then Definition \ref{def-Ekk} is equivalent to
\begin{equation*}
E_{kk}\cdot v_S=E_{kk}^\prime v_S-E_{kk}^\prime |0\rangle .
\end{equation*}

\subsubsection{Dirac's Electron Sea} \label{sec-dirac}
The infinite wedge and the action of $E_{kk}^\prime$ have a physical interpretation known as Dirac's electron sea.  Here, the vector space $V$ represents the possible energy levels of a single electron according to quantum mechanics.  Pauli's exclusion principle states that two electrons cannot occupy the same energy level, giving rough motivation that recording a collection of electrons should be done with wedges of $V$.

Difficulty arose with the negative energy vectors; although they were predicted by quantum mechanics, it doesn't make sense for an electron to have to have negative energy.  Dirac's solution to this problem was to redefine the vacuum: rather than the vacuum consisting of no electrons at all, it consists of an ``infinite sea'' of negative energy electrons, with every possible negative energy level filled - this is precisely what corresponds to our vacuum vector $|0\rangle$.

Then it is possible to have a hole in the sea: for instance, the vector $\underline{-1/2}\wedge \underline{-5/2}\wedge \underline{-7/2}\wedge\cdots$ with every negative energy state filled except for $-3/2$.  Since we are ``missing'' an electron with energy $-3/2$, this should be the same thing as having a particle with energy $3/2$ but charge opposite that of an electron: a positron.  Thus, the vectors $v_S$ correspond to pure states consisting a finite number of electrons, with energies $S\cap \Z^+_{1/2}$, together with a finite number of positrons, with energies $S^c\cap \Z^-_{1/2}$.

\subsubsection{} \label{sec-CnE}

This notation explains the names of some common diagonal operators on the infinite wedge.

\begin{definition}
The \emph{charge operator} $C$ is defined by
\begin{equation*}
C=\sum_{k\in\Z+1/2} E_{kk}
\end{equation*}
The \emph{energy operator} $E$ is defined by
\begin{equation*}
E=\sum_{k\in\Z+1/2} k E_{kk}
\end{equation*}
\end{definition}

The charge and energy operator are each diagonal in the basis $v_S$, and the eigenvalues corresponding to the charge and energy of the corresponding collection of particles.

In particular, we say a vector $v_S\in \infwedge V$ has charge $c$ (or energy $e$), if $C\cdot v_S=0$ (respectively, $E\cdot v_S)$.  This examples the name of the charge 0 subspace $\infwedge_0 V$ -- it consists of this vectors annihilated by $C$.

An operator $M$ has energy $e$ if acts by decreasing the energy by $e$, that is $[E,M]=-e$.  The operator $\alpha_n$ has energy $n$.

It will be important to us that operators with positive energy annihilate the vacuum.  If $M$ is any operator, and $P$ is an operator with positive energy $\langle MP\rangle=0$.  Taking the adjoint, if $N$ is an operator with negative energy, $\langle NM\rangle=0$.

\subsubsection{}
For a vector $v_\lambda\in\infwedge_0 V$, viewing its Maya diagram in terms of electrons and positrons as above corresponds to viewing the partition in Frobenius notation.

Frobenius notation describes a partition $\lambda$ as an ordered pairs of strictly decreasing lists of non-negative integers $(a_1,\dots, a_n)$ and $(b_1,\dots, b_n)$, with $|\lambda|=n+\sum a_i+\sum b_i$.  The number $n$ and the $n$-tuples $a_i$ and $b_i$ are obtained as follows: in Russian notation, there will be some number $n$ of boxes lying directly above 0; number them so that the first box is at the origin, the last box is on the border.  Then $a_i$ is the number of boxes above and to the left of the $i$th box, while $b_i$ is the number of boxes above and to the right; they obviously satisfy the properties listed.

The electron and positron view of the infinite wedge coincides most closely with \emph{Modified Frobenius notation} rather than remove the $n$ boxes, we cut them in half, and adjoin each half to the strip of boxes making up $a_i$ and $b_i$; obtaining half integers $a_i^\prime=a_i+1/2$ and $b_i^\prime=b_i+1/2$, with $a_i^\prime,b_i^\prime$ strictly decreasing $n$-tuples in $\Z_{1/2}^+$, and \begin{equation} \label{blahblah}
\sum a_i^\prime+\sum b_i^\prime=|\lambda|.
\end{equation}
  From this description it is immediate that the $a^\prime_i$ are the energy levels of the electrons appearing in the Maya diagram of $v_\lambda$, and the $b^\prime_i$ are the energy levels of the positrons appearing in $v_\lambda$.

Furthermore, with $E=\sum k E_{k,k}$ the energy operator, Equation (\ref{blahblah}) immediately becomes
\begin{equation*}
E \cdot v_\lambda=|\lambda| v_\lambda.
\end{equation*}

\subsubsection{}
Frobenius gave a formula for the characters of the symmetric group evaluated on a conjugacy class that in the case the conjugacy class is a transposition has the following elegant expression in terms of Frobenius notation:
\begin{Frobtheorem}
\begin{equation}
f_2(\lambda)=\frac{1}{2}\sum_{i=1}^n a_i(a_i-1)-\frac{1}{2}\sum_{i=1}^n b_i(b_i-1).
\end{equation}
\end{Frobtheorem}

In terms of modified Frobenius notation this becomes
\begin{equation} \label{modifiedfrob}
f_2(\lambda)=\frac{1}{2}\sum_{1=1}^n (a^\prime_i)^2-\frac{1}{2}\sum_{i=1}^n (b_i^\prime)^2.
\end{equation}
Note that while $a_i(a_i-1)=(a^\prime_i)^2-1/4$, the occurrence of $1/4$ here will cancel with a similar occurrence in the sum over the $b^\prime_i$, giving Equation (\ref{modifiedfrob}).

This is conveniently encoded in terms of the infinite wedge. Define the operator $\mathcal{F}_2$ by:
\begin{definition} \label{def-F}
\begin{equation*}
\mathcal{F}_2=\sum_{k\in \Z+1/2} \frac{k^2}{2} E_{k,k}.
\end{equation*}
\end{definition}
Equation (\ref{modifiedfrob}) immediately gives
\begin{lemma} \label{lem-Fwedge}
\begin{equation*}
\mathcal{F}_2\cdot v_\lambda=f_2(\lambda)v_\lambda.
\end{equation*}
\end{lemma}
That is, the $v_\lambda$ form an eigenbasis for $\mathcal{F}_2$, with eigenvalues $f_2(\lambda)$.

\subsection{Hurwitz numbers as vacuum expectations}
\subsubsection{}
We are now in a position to write Hurwitz numbers as vacuum expectations on the infinite wedge.

\begin{theorem} \label{thm-Hwedge}
\begin{equation*}
H_r(\mu,\nu)=\frac{1}{\prod \mu_i}\frac{1}{\prod \nu_j}\left\langle \prod \alpha_{\mu_i} \mathcal{F}_2^r \prod \alpha_{-\nu_j}\right\rangle .
\end{equation*}
\end{theorem}
\begin{proof}
This is nothing but an encoding of Equation (\ref{eq-HurwitzRep}) using Lemmas \ref{lem-MNwedge} and \ref{lem-Fwedge}.  First expand $\prod \alpha_{\mu_i}|0\rangle$ into the $v_\lambda$ basis, with $v_\lambda$ appearing with a factor of $\chi_\nu^\lambda$. In the $v_\lambda$ basis,  $\mathcal{F}_2^r$ is diagonal and produces the factor of $f_2(\lambda)^r$.  Finally, applying $\prod \alpha_{-\nu_j}$ to $v_\lambda$ produces $\chi_\mu^\lambda|0\rangle$.
\end{proof}

Theorem \ref{thm-Hwedge} appears in \cite{OHur}, and nearly in \cite{D}.

Theorem \ref{thm-Hwedge} leads immediately to a vacuum expectation expression for the $m+n$ series:
\begin{equation} \label{eq-basic}
H_{\mu,\nu}(z)=\frac{1}{\prod \mu_i}\frac{1}{\prod \nu_j}\left\langle \prod \alpha_{\mu_i} e^{z\mathcal{F}_2} \prod \alpha_{-\nu_j}\right\rangle.
\end{equation}

\subsection{The operators $\E_r(z)$}
Though an attractive formula, as presented so far Theorem \ref{thm-Hwedge} is essentially nothing but the classical Equation (\ref{eq-HurwitzRep}).  The benefit of writing it in terms of the infinite wedge will come from manipulating this expression in terms of commutators of the operators involved.  In this section, we will begin that process, by rewriting Equation (\ref{eq-basic}) for $H_{\mu,\nu}(z)$ in terms of Okounkov and Pandharipande's operators $\E_r(z)$.

\begin{definition}
\begin{equation*}
\E_r(z)=\sum_{k\in\Z+1/2} e^{kz}E_{k,k+r}+\frac{\delta_{r,0}}{\varsigma(z)}.
\end{equation*}
\end{definition}
 The commutators of these operators will produce the $\varsigma$ terms in Theorem \ref{thm-main}.

\subsubsection{}
The $\frac{\delta_{r,0}}{\varsigma(z)}$ term can be understood as adding back the infinite sum that would have appeared if we had used the naive action on the infinite edge, and that we regularized away with Definition \ref{def-Ekk}.  The naive action of the first term of $\E_0(z)$ on the vacuum $|0\rangle$ is an infinite sum that converges for $z>0$:
\begin{eqnarray*}
e^{-z/2}+e^{-3z/2}+e^{-5z/2}+\cdots &=&e^{-z/2}\left(1+e^{-z}+e^{-2z}+\cdots \right) \\
&=& \frac{e^{-z/2}}{1-e^{-z}}=\frac{1}{\varsigma(z)}.
\end{eqnarray*}
So, for $z>0$ we have
\begin{equation*}
\E_0(z)=\sum_{k\in\Z+1/2} e^{kz} E_{k,k}^\prime .
\end{equation*}

\subsubsection{}
We now collect the basic facts about $\E_r(z)$ that we will require.  All of them can be proved by straightforward calculations which we omit.

The operator $\E_r(z)$ has energy $r$ and for $r\neq 0$ is closely related to the operators $\alpha_r$.  It is immediate from the definition that
\begin{equation} \label{eq-E0}
\E_r(0)=\alpha_r.
\end{equation}

In fact, $\E_r(z)$ is equal to $\alpha_r$ conjugated by a zero energy operator depending on $z$, as in the following identity from \cite{OP2}:
\begin{equation} \label{eq-Econj}
e^{z\mathcal{F}_2}\alpha_{-n}e^{-z\mathcal{F}_2}=\E_{-n}(nz).
\end{equation}

Finally, the following commutator is the reason we introduce the $\mathcal{E}_r(z)$, and will be our main tool in Section \ref{sec-mainformula}
\begin{equation} \label{eq-Ecommutator}
[\E_r(z),\E_s(w)]=\varsigma\left(\text{det}\begin{bmatrix} r & z \\ s & w \end{bmatrix}\right)\E_{r+s}(z+w).
\end{equation}

\subsubsection{}

We will find it convenient to rewrite Equation (\ref{eq-basic}) in terms of the operators $\E_r(z)$, essentially following \cite{OP2}.

Since $\mathcal{F}_2$ annihilates the vacuum, we may rewrite Equation (\ref{eq-basic}) as
\begin{equation*}
H_{\mu,\nu}(z)=\frac{1}{\prod \mu_i}\frac{1}{\prod \nu_j}\left\langle \prod \alpha_{\mu_i} \prod  \left( e^{z\mathcal{F}_2} \alpha_{-\nu_j} e^{-z\mathcal{F}_2}\right)\right\rangle.
\end{equation*}
Applying Equations (\ref{eq-E0}) and (\ref{eq-Econj}) then immediately gives
\begin{equation} \label{withEs}
H_{\mu,\nu}(z)=\frac{1}{\prod \mu_i}\frac{1}{\prod \nu_j}\left\langle \prod \E_{\mu_i}(0) \prod \E_{-\nu_j}(z\nu_j)\right\rangle.
\end{equation}

\section{The main formula} \label{sec-mainformula}

To derive our formula for the $n+m$ point series, we follow a standard strategy for computing vacuum expectations: successively commute operators with positive energy to the right. For example, a similar case is the use in \cite{OP2} to compute the $n$-point Gromov-Witten invariants. 
\subsubsection{}
Before deriving our main formula, we introduce a few pieces of new notation so that we may express it cleanly.

\begin{definition} \label{def-compacte}
For $I\subset[m]$ and $J\subset [n]$, define
\begin{equation*}
\EE{I}{J}=\E_{|\mu_I|-|\nu_J|}(z|\nu_J|).
\end{equation*}
\end{definition}

\begin{definition} \label{def-bvs}
For $I,K\subset [m]$ and $J,L \subset [n]$, define
\begin{equation*}
\bvs{I}{J}{K}{L}=\varsigma\left(\text{det}\begin{bmatrix} |\mu_I|-|\nu_J| & z|\nu_J| \\ |\mu_K|-|\nu_L| & z|\nu_L| \end{bmatrix}\right)=
\varsigma\left(z\cdot \text{det}\begin{bmatrix} |\mu_I| & |\nu_J| \\ |\mu_K| & |\nu_L| \end{bmatrix}\right),
\end{equation*}
\end{definition}

Definitions \ref{def-compacte} and \ref{def-bvs} are such that for $I,K\subset [m]$ disjoint and $J,L\subset [n]$ disjoint Equation (\ref{eq-Ecommutator}) becomes
\begin{equation}
\left[\EE{I}{J},\EE{K}{L}\right]=\bvs{I}{J}{K}{L}\EE{I\cup K}{J\cup L}
\end{equation}
and Equation (\ref{withEs}) becomes
\begin{equation} \label{withEEs}
H_{\mu,\nu}(z)=\frac{1}{\prod \mu_i}\frac{1}{\prod \nu_j}\left\langle \prod_{i=1}^m \EE{i}{\emptyset} \prod_{j=1}^n \EE{\emptyset}{j}\right\rangle.
\end{equation}

\subsection{The algorithm}

In this section, we describe an algorithm that computes the $n+m$ point series from Equation (\ref{withEEs}).  The algorithm leads immediately to Theorem \ref{maintheorem}.

\subsubsection{Inductive Step}
At each step, we will have a sum of vacuum expectations of operators of the form
\begin{equation*}
\left\langle \prod_{i=1}^k \EE{I_i}{J_i}\right\rangle.
\end{equation*}

We assume that $\mu,\nu$ are in the interior of a chamber $\mathfrak{c}$ of the resonance arrangement, so that the $\EE{I_i}{J_i}$ cannot have zero energy, except for $\EE{[m]}{[n]}$.

The algorithm proceeds as follows: pick any term of the sum, and find the rightmost $\E$ with positive energy; say, $\EE{I}{J}$.  If this is the rightmost $\E$ in total, then that term contributes zero, since operators with positive energy annihilate the vacuum.

Otherwise, we commute that term with the operator immediately to its right, which we will denote $\EE{K}{L}$.  That is, substitute
\begin{eqnarray*}
\EE{I}{J}\EE{K}{L}&=&
\EE{K}{L}\EE{I}{J}+\left[\EE{I}{J},\EE{K}{L}\right]
\\
&=&\EE{K}{L}\EE{I}{J}+
\bvs{I}{J}{K}{L}\EE{I\cup J}{K\cup L}
\end{eqnarray*}
where we have used Equation \ref{eq-Ecommutator}.

\subsubsection{Termination}

The substitution gives us two terms: one with one less $\E$ and an additional $\varsigma$ factor, which we call the {\em canceling term}, and one with a positive energy $\E$ further to the right, which we call the {\em passing term}.

Iterating this procedure, eventually either a positive energy $\E$ is at the far right, in which case the term contributes zero, or there is only a single $\E$ term remaining, which must then be $\EE{[m]}{[n]}=\E_0(dz)$, which has vacuum expectation
\begin{equation*}
\left\langle \E_0(dz) \right\rangle=\frac{1}{\varsigma(dz)}.
\end{equation*}

\subsubsection{}

To extract a formula from the algorithm, we need to record all possible ways of taking passing and canceling terms to get a nonzero result.  It suffices to record only the sets that are involved in the canceling terms.

\begin{definition} For $|\mu|=|\nu|$ partitions with ordered parts, a \emph{commutation pattern} $P$ consists of four $m+n-1$-tuples of sets $I^P_\ell, J^P_\ell, K^P_\ell, L^P_\ell$,
\begin{equation*}
I^P_\ell, K^P_\ell\subset [m], J^P_\ell, L^P_\ell\subset [n], 1\leq \ell\leq n+m-1
\end{equation*}
such that the algorithm has a nonvanishing term where the $k$th commutator computed was $[\EE{I^P_k}{J^P_k}, \EE{K^P_k}{L^P_k}]$.

We denote the set of all possible commutation patterns $P$ for a given $\mu,\nu$ by $CP(\mu,\nu)$.
\end{definition}

The algorithm then immediately gives our main theorem.
\begin{theorem} \label{thm-main}
\begin{equation*}
H_{\mu,\nu}(z)=\frac{1}{\prod \mu_i}\frac{1}{\prod \nu_j}\frac{1}{\varsigma(dz)} \sum_{P\in CP(\mu,\nu)} \prod_{\ell=1}^{n+m-1} \bvs{I^P_\ell}{J^P_\ell}{K^P_\ell}{L^P_\ell}.
\end{equation*}
\end{theorem}

\subsection{Examples}
\subsubsection{} Reproducing \cite{GJV}.
\begin{example} \ref{example-GJV}
We compute the $1+n$ series $H_{(d),\nu}(z)$, reproducing the calculation in \cite{GJV}.  The key observation is that $CP((d),\nu)$ consists of a single element: if a passing term is ever taken the result will be zero, as there would be a negative energy term on the far left. That is, we want to compute the vacuum expectation
\begin{equation*}
\left\langle \E_d(0) \prod_{j=1}^m \E_{-\mu_j}(z\mu_j)\right\rangle.
\end{equation*}
The only operator with positive energy is $\E_d(0)$, and so if we have ever take the passing term one of the negative energy operators $\E_{-\mu_j}(z\mu_j)$ would be at the far left, and the result would be zero.

Thus, we have that $I_\ell=\{1\}, J_k=[\ell-1]$, where $[0]=\emptyset$, $K_\ell=\emptyset$ and $L_\ell=\{\ell\}$, and so the $\ell$th commutator produces a factor of:
\begin{equation*}
\bvs{I^P_\ell}{J^P_\ell}{K^P_\ell}{L^P_\ell}
=\varsigma \left(z\cdot \text{det}\begin{bmatrix} d & \sum_{i=1}^{\ell-1} \nu_i \\ 0 & \nu_\ell \end{bmatrix}\right)
=\varsigma(zd\nu_\ell),
\end{equation*}
and so in this case Theorem \ref{thm-main} is equivalent to Equation (\ref{GJVFormula}).
\end{example}

\subsubsection{} \label{sec-ordering} The next two examples compute the 2+2-point series with two different orderings of the $\mu_i$ and $\nu_j$, giving two different looking answers that are equivalent by an identity for $\varsigma$.  We will assume $\mu_1>\nu_1, \nu_2>\mu_2$.

\begin{example} \label{example-221}
We will compute the series $H_{(\mu_1,\mu_2), (\nu_1, \nu_2)}$. From Equation \ref{withEs}, we have  vacuum expectation we are interested in is then
\begin{equation*}
H_{\mu, \nu}(z)=\frac{1}{\mu_1\mu_2\nu_1\nu_2}\langle \E_{\mu_1}(0)\E_{\mu_2}(0)\E_{-\nu_1}(z\nu_1)\E_{-\nu_2}(z\nu_2) \rangle
\end{equation*}

The rightmost $\E$ with positive energy is $\E_{\mu_2}(0)$, so after substituting we get
\begin{equation*}
\langle \E_{\mu_2}(0)\E_{-\nu_1}(z\nu_1)\E_{\mu_2}(0)\E_{-\nu_2}(z\nu_2) \rangle + \varsigma(z\mu_2\nu_1) \langle \E_{\mu_1}(0)\E_{\mu_2-\nu_1}(z\nu_1)\E_{-\nu_2}(z\nu_2)\rangle.
\end{equation*}

Carrying out the inductive step again on the first term of the sum, we the rightmost positive energy term is again $\E_{\mu_2}(0)$.  If we take the passing term again, $\E_{\mu_2}(0)$ would be at the far right of the whole term and annihilate the vacuum, and so we must keep only the canceling term: \begin{equation*}
\varsigma(z\mu_2\nu_2) \langle \E_{\mu_1}(0)\E_{-\nu_1}(z\nu_1)\E_{\mu_2-\nu_2}(z\nu_2) \rangle
\end{equation*}
Note that this has only one term of positive energy, $\E_{\mu_1}(0)$, and so as we continue the algorithm it is immediate that any term where we take a passing term will be zero.  Therefore, we see:

\begin{eqnarray*}
 \langle \E_{\mu_1}(0)\E_{-\nu_1}(z\nu_1)\E_{\mu_2-\nu_2}(z\nu_2) \rangle
&=& \varsigma(z\mu_1\nu_1) \langle \E_{\mu_1-\nu_1}(z\nu_1)\E_{\mu_2-\nu_2}(z\nu_2) \rangle \\
&=&\frac{ \varsigma(z\mu_1\nu_1)\varsigma(z(\mu_1\nu_2-\mu_2\nu_1))}{\varsigma(dz)}
\end{eqnarray*}
and this term gives a total contribution of
\begin{equation*}
\frac{\varsigma(z\nu_2\mu_2) \varsigma(z\mu_1\nu_1)\varsigma(z(\mu_1\nu_2-\mu_2\nu_1))}{\varsigma(dz)}.
\end{equation*}

Similarly, returning to the first canceling term, there is only one term with negative energy, and so that term is
\begin{equation*}
\langle \E_{\mu_1}(0)\E_{\mu_2-\nu_1}(z\nu_1)\E_{-\nu_2}(z\nu_2)\rangle=\frac{\varsigma(\mu_1\nu_1z)\varsigma(d\nu_2z)}{\varsigma(dz)}.
\end{equation*}

Putting everything together, we have
\begin{equation} \label{eq-ugly22}
H_{\mu,\nu}(z)=\frac{\varsigma(z\nu_2\mu_2) \varsigma(z\mu_1\nu_1)\varsigma(z(\mu_1\nu_2-\mu_2\nu_1))+\varsigma(z\mu_2\nu_1)\varsigma(\mu_1\nu_1z)\varsigma(d\nu_2z)}{\mu_1\mu_2\nu_1\nu_2\varsigma(dz)}.
\end{equation}

\end{example}

Note that the series $H_{\mu, \nu}(z)$ depends only on $\mu$ and $\nu$ as unordered partitions, but our algorithm depends on the ordering chosen for the elements of $\mu$ and $\nu$.  Different choices of orderings give different expressions for $H_{\mu,\nu}(z)$.  These different expressions can be seen to describe the same series by applications of identities for $\varsigma(z)$.  We illustrate this in the next example by calculating the $2+2$ series again, but with a different ordering of the partitions.

\begin{example} \label{example-222}
Let us repeat the calculation of the 2+2-point series from Example \ref{example-221} with the other ordering of the parts of $\mu$:
\begin{equation*}
H_{\mu, \nu}(z)=\frac{1}{\mu_2\mu_1\nu_1\nu_2}\langle \E_{\mu_2}(0)\E_{\mu_1}(0)\E_{-\nu_1}(z\nu_1)\E_{-\nu_2}(z\nu_2) \rangle.
\end{equation*}
At the first step, the right most term negative energy term is $\E_{\mu_1}(0)$; we see that the passing term is zero, as the rightmost two operators in the passing term are $\E_{\mu_2}(0)\E_{-\nu_1}(z\nu_1)$, which has energy $\mu_2-\nu_1<0$.

In fact, the only nonzero commutation pattern is when all canceling terms are taken, and so the algorithm gives:
\begin{equation} \label{eq-nice22}
H_{\mu,\nu}(z)=\frac{1}{\mu_1\mu_2\nu_1\nu_2}\frac{\varsigma(z\mu_1\nu_1)\varsigma(z\mu_1\nu_2)\varsigma(z\mu_2d)}{\varsigma(zd)}.
\end{equation}
\end{example}

These two expressions for the 2+2-point series -- i.e., Equations (\ref{eq-ugly22}) and (\ref{eq-nice22}) -- are easily checked to be equivalent using $|\mu|=d=|\nu|$ and a simple identity for $\varsigma$, which can be written attractively as:
\begin{equation} \label{eq-cyclicsigmaidentity}
\varsigma(a-b)\varsigma(c)+\varsigma(b-c)\varsigma(a)+\varsigma(c-a)\varsigma(b)=0.
\end{equation}

\section{Consequences} \label{sec-consequences}

In this section, we explore some of the consequences of Theorem \ref{thm-main}.  Section \ref{sec-SPP} proves the Strong Piecewise Polynomiality conjecture, as well as Corollary \ref{bernoullicorollary}; Section \ref{sec-wallcrossing} derives a wall crossing formula for the piecewise polynomials, and Section \ref{sec-specialchambers} determines those chambers on which Theorem \ref{thm-main} has a particular nice form.

\subsection{Strong Piecewise Polynomiality} \label{sec-SPP}

\begin{lemma} \label{lem-CPwalls}
The set $CP(\mu,\nu)$ depends only on the chamber of $(\mu_,\nu)$ in the resonance arrangement $R_{m,n}$.
\end{lemma}

\begin{proof}
Looking at decisions involved in the algorithm, the only way $\mu$ and $\nu$ enter into the decisions is determine whether a given $\EE{I}{J}$ has positive or negative energy.  But the energy of $\EE{I}{J}$ is $-|\mu_I|+|\nu_J|$, and knowing the chamber of $(\mu,\nu)$ in $R_{m,n}$ is equivalent to knowing the signs of all $-|\mu_I|+|\nu_J|$.
\end{proof}
                                              
From Lemma \ref{lem-CPwalls}, it follows that in a given chamber of the resonance arrangement, the sum in Theorem \ref{thm-main} contains exactly the same terms.  In the remainder of Section \ref{sec-SPP}, we will prove Corollary \ref{maincorollaries}, by showing that the statement holds true for each individual term in the sum appearing in Theorem \ref{thm-main}.

\subsubsection{Polynomiality}

From the form of Theorem \ref{thm-main}, it is only clear that each term is a Laurent polynomial: we must divide by the initial factor of $\prod \mu_i \prod \nu_j$, and additionally we must divide by a factor of $d$ to invert $1/\varsigma(zd)$.  Thus, to show that $H_g(\mu,\nu)$ is a polynomial, we must show that $d\prod \mu_i\prod\nu_i$ divides
\begin{equation} \label{eq-prod}
\prod_{\ell=1}^{n+m-1} \bvs{I^P_\ell}{J^P_\ell}{K^P_\ell}{L^P_\ell}
\end{equation}
for each $P\in CP(\mu,\nu)$.

We first show that (\ref{eq-prod}) is divisible by $\mu_i$.  The term $\mu_i$ enters the calculation through $\EE{i}{\emptyset}$.  This term has positive energy, and so eventually must occur as the first entry of some commutator; suppose this is the $\ell$th commutator, and so $I^P_\ell=\{i\}$, and $J^P_\ell=\emptyset$.

Then
\begin{equation*}
\bvs{I^P_\ell}{J^P_\ell}{K^P_\ell}{L^P_\ell}=\varsigma \left(z\cdot \text{det} \begin{bmatrix} \mu_i & 0 \\ X & Y \end{bmatrix} \right)=\varsigma(z\mu_i Y)
\end{equation*}
where $X$ and $Y$ are some linear functions of the parts of $\mu$ and $\nu$.   so we have seen that $\mu_i$ divides $\bvs{I^P_\ell}{J^P_\ell}{K^P_\ell}{L^P_\ell}$.

A completely analogous argument shows that there must be some $\ell$ with $K^P_\ell=\emptyset$ and $L^P_\ell=\{j\}$, and that for this $\ell,$ $\nu_j$ divides $\bvs{I^P_\ell}{J^P_\ell}{K^P_\ell}{L^P_\ell}$.

Finally, we show (\ref{eq-prod}) is divisible by $d$.  Consider the last commutator, when $\ell=m+n-1$.  Every $\mu_i$ and $\nu_j$ is involved in this commutator, so $I^P_\ell\cup K^P_\ell=[m]$ and $J^P_\ell\cup L^P_\ell=[n]$. Thus, adding the two rows of the matrix appearing in the definition of $\bvs{I^P_\ell}{J^P_\ell}{K^P_\ell}{L^P_\ell}$ gives the vector $(d,d)$, and so $\bvs{I^P_\ell}{J^P_\ell}{K^P_\ell}{L^P_\ell}$ is divisible by $d$.

We have shown that each factor of $d\prod\mu_i\prod\nu_j$ divides (\ref{eq-prod}), and examining the argument it is evident that their product does.
\subsubsection{Nonzero degrees}

Because $\varsigma(z)$ is odd, we immediately see that the polynomials are either odd or even.

To determine the degree of the polynomial, note that in $\bvs{I^P_\ell}{J^P_\ell}{K^P_\ell}{L^P_\ell}, z$ is multiplied by a quadratic polynomial in the $\mu_i$ and $\nu_j$, while in $\frac{1}{\varsigma(dz)}$, $z$ is multiplied by a linear function of the $\mu$ and $\nu$.  Therefore, the highest coefficient of the polynomial will occur when as many as possible of the $z$'s come from the first term, and the lowest degree term will occur when all come from the second.

The genus $g$ Hurwitz number appears as the coefficient of $z^{2g-2+m+n}$.  As we will always divide by one $z$ from $\frac{1}{\varsigma(dz)}$, the highest degree term will come from the coefficient of $z^{2g-1+m+n}$ from the product.  As each of these $z$'s multiplies a quadratic function, this gives us a polynomial of degree $4g-2+2m+2n$, and when we factor out of the $m+n+1$ linear factors of $\mu_i,\nu_j$ and $d$, what remains is a polynomial of degree $4g-3+m+n$.

For the lowest degree term, note that we must take at least one $z$ from each of the $m+n-1$ factors inside; this gives us a polynomial of degree $2m+2n-2$.  Then, we must take the remaining $z^{2g}$ from $\frac{1}{\varsigma(dz)}$, each of which gives us a linear factor, for a polynomial of degree $2g+2m+2n-2$, which when we cancel the $m+n+1$ linear factors gives a polynomial of degree $2g-3+m+n$.

\subsubsection{Positivity}

First, note the alternating signs in $\varsigma(z)$ and in 
\begin{equation*}
\frac{1}{\varsigma(z)}=\frac{1}{z}-\sum_{n=1}^\infty \frac{(1-\frac{1}{2}^{2n-1})B_{2n}z^{2n-1}}{(2n)!},
\end{equation*}
since the Bernoulli numbers $B_{2n}$ alternate sign.

From these alternating signs, the positivity result will follow if each determinant appearing in $\bvs{I}{J}{K}{L}$ is positive, i.e. if each
\begin{equation} \label{eq-determinant}
\det \begin{bmatrix} |\mu_I|-|\nu_J| & |\nu_J| \\ |\mu_K|-|\nu_L| & |\nu_L| \end{bmatrix}
\end{equation} is positive.

Recall that $\bvs{I}{J}{K}{L}$ appears in the commutator $[\EE{I}{J}, \EE{K}{L}]$, and that this commutator only occurs in the algorithm if $\EE{I}{J}$ has positive energy and $\EE{K}{L}$ has negative energy.  But this implies the first row of the matrix in (\ref{eq-determinant}) is in the first quadrant and the second row is in the second quadrant, and so the determinant is positive.

Corollary \ref{bernoullicorollary} also follows immediately from the expansion of $\frac{1}{\varsigma(dz)}$.

\subsection{Wall Crossing} \label{sec-wallcrossing}

We now prove the wall crossing formula, Theorem \ref{thm-wallcrossing}.

Suppose that $\mathfrak{c}_1$ and $\mathfrak{c}_2$ are two chambers bordering along the wall $W_{I,J}$.  The main idea is that $CP(\mathfrak{c}_1)$ and $CP(\mathfrak{c}_2)$ are nearly the same, and that the permutation patterns that are in one but not the other have a nice description.

\subsubsection{Reduction to commutation patterns containing $\EE{I}{J}$}
We begin by noting that the algorithm will run nearly the same on both $\mathfrak{c}_1$ and $\mathfrak{c}_2$; equivalently, the set $CP(\mu,\nu)$ will change in an easily described way as we cross the wall.

As noted before, the only way the algorithm depends on $\mu,\nu$ is checking if the operators $\EE{K}{L}$ have positive or negative energy; this is equivalent to knowing which side of the wall $W_{K,L}$ we are on.  Hence, if a commutation pattern $P\in CP(\mathfrak{c}_1)$ does not ever produce the operator $\EE{I}{J}$, then this commutation pattern $P$ will also appear in $CP(\mathfrak{c}_2)$, and so the resulting terms will cancel in the formula for $WC^{I,J}_{\mu,\nu}(z)$.

Thus, it is effective to choose an ordering of the $\mu_I$ and $\nu_J$ so that the elements of $I$ and $J$ occur in the middle; i.e., with a starting vacuum expectation is of the form:
\begin{equation*}
\left\langle \prod_{i\notin I} \EE{i}{\emptyset}\prod_{i\in I} \EE{i}{\emptyset} \prod_{j\in J} \EE{\emptyset}{j} \prod_{j\notin J} \EE{\emptyset}{j} \right\rangle .
\end{equation*}
Then if a commutation pattern produces $\EE{I}{J}$, the first vacuum expectation to contain $\EE{I}{J}$ must be exactly
\begin{equation} \label{chokepoint}
\left\langle \left( \prod_{i\notin I} \EE{i}{\emptyset}\right) \EE{I}{J} \left(\prod_{j\notin J} \EE{\emptyset}{j} \right) \right\rangle .
\end{equation}

\subsubsection{Contribution before producing $\EE{I}{J}$}
Up until the vacuum expectation (\ref{chokepoint}) is produced, the algorithm will have ran identically on both $\mathfrak{c}_1$ and $\mathfrak{c}_2$.  Let $T_1$ be the product of $\varsigma$ terms the algorithm produces in reaching (\ref{chokepoint}) up this point; we will now show that $T_1$ is essentially $H_{\mu_I, \nu_J+\delta}(z)$.  

The vacuum expectation involved in computing $H_{\mu_I, \nu_J+\delta}(z)$ is
\begin{equation*}
\left\langle \prod_{i\in I} \EE{i}{\emptyset} \prod_{j\in J}\EE{\emptyset}{j}\E_{-\delta}(\delta z) \right\rangle.
\end{equation*}
The key observation is that the $\E_{-\delta}(\delta z)$ term cannot be involved in a commutator leading to a nonzero term until the very last commutator.  Suppose there were a nonzero term where it was involved in a commutator with $\EE{K}{L}$ for some $K\subset I$ and $L\subset J$, with at least one subset being proper.  Then we must have had that $\mu_K-\nu_L$ to be positive, but $\mu_K-\nu_L-\delta$ to be negative.  However, this contradicts the fact that we are in a chamber which borders $\delta=0$.

Thus, all the other commutators must be computed first, which produces the factor $T_1$:

\begin{eqnarray*}
\left\langle \prod_{i\in I} \EE{i}{\emptyset} \prod_{j\in J}\EE{\emptyset}{j}\E_{-\delta}(\delta z) \right\rangle
&=&T_1\left\langle  \EE{I}{J}\E_{-\delta}(z\delta) \right\rangle \\
&=&T_1\varsigma(z\delta d_1)\left\langle\E_0(zd_1)\right\rangle \\
&=& T_1 \frac{\varsigma(z\delta d_1)}{\varsigma(zd_1)}
\end{eqnarray*}
where we have used
\begin{equation*}
\EE{I}{J}=\E_{|\mu_I|-|\nu_J|}(z|\nu_J|)=\E_{\delta}(z(d_1-\delta)).
\end{equation*}

Including the factors of $\mu_i, \delta,$ and $\nu_j$ involved, we have that
\begin{equation} \label{eq-T1}
T_1=\delta\prod_{i\in I}\mu_i \prod_{j\in J} \nu_j\frac{ \varsigma(z d_1)}{\varsigma(z\delta d_1)} H_{\mu_I, \nu_J+\delta}(z).
\end{equation}

\subsubsection{Contribution after producing $\EE{I}{J}$}
To finish the proof, we must compute the difference of the vacuum expectation (\ref{chokepoint}) on chambers $\mathfrak{c}_2$ and $\mathfrak{c}_1$; we will denote this series $T_2$. If we ran the algorithm as usual on each chamber, the results would diverge at the very first step as the central term $\EE{I}{J}$ would head it opposite directions: it has energy $\delta$, which is negative on $\mathfrak{c}_1$ and positive on $\mathfrak{c}_2$.

To better compare the vacuum expectation (\ref{chokepoint}) on the two chambers, we will follow the algorithm as usual on $\mathfrak{c}_1$, but commute $\EE{I}{J}$ to the left in $\mathfrak{c}_2$, even though it has positive energy.

Note that if $\EE{I}{J}$ is involved in a canceling term as we move it to the left, then afterwards the algorithm will run naturally on both sides of the wall from that point on: after the cancelation, we will have terms of the form $\EE{I\cup K}{J\cup L}$, where $K, L$ are not both empty; and since the chambers differed by only one wall, the sign of the energy of this term will be the same independent of which side of the wall we were on.  Thus, the contributions of these terms on the two sides of the wall will cancel.

However, if $\EE{I}{J}$ reaches the far left end, then the result will vanish on $\mathfrak{c}_1$, but not on $\mathfrak{c}_2$.
Thus, we see that
\begin{equation} \label{afterchokepoint}
T_2=\left\langle \E_{\delta}(z|\nu_J|) \left(\prod_{i\notin I} \EE{i}{\emptyset}\right) \left( \prod_{j\notin J} \EE{\emptyset}{j}\right)  \right\rangle
\end{equation}
evaluated on $\mathfrak{c}_1$.

The vacuum expectation (\ref{afterchokepoint}) is very nearly the vacuum expectation that would appear to calculate $H_{\mu_{I^c}+\delta, \nu_{J^c}}(z)$.  The only difference is that the leftmost $\E$ term appearing is $\E_{\delta}(z|\nu_J|)$ instead of $\E_{\delta}(0)$.

However, this term can only be involved in the very last commutator.  All other $\E$ terms must cancel first, producing $\EE{I^c}{J^c}=\E_{-\delta}(z|\nu_J^c|)$.  So, the last step of the algorithm for vacuum expectation (\ref{afterchokepoint}) will end with
\begin{eqnarray*}
\left \langle \E_{\delta}(z|\nu_J|) \E_{-\delta}(z|\nu_J^c|)\right\rangle
= \frac{\varsigma(zd\delta)}{\varsigma(zd)}
\end{eqnarray*}
instead of
\begin{eqnarray*}
\left \langle \E_{\delta}(0) \E_{-\delta}(z|\nu_J^c|)\right\rangle
=\frac{\varsigma(z\delta d_2)}{\varsigma(z d_2)}.
\end{eqnarray*}

Thus, we have that
\begin{equation} \label{eq-T2}
T_2=\delta\prod_{i\notin I}\mu_i \prod_{j\notin J} \nu_j\frac{\varsigma(z d\delta)}{\varsigma(zd)}\frac{\varsigma(z d_2)}{\varsigma(z d_2 \delta)} H_{\mu_I^c+\delta,\nu_J^c}(z).
\end{equation}

Since $WC^{I,J}_{\mu,\nu}(z)=\frac{T_1T_2}{\prod \mu_i \prod \nu_j}$, Equations (\ref{eq-T1}) and (\ref{eq-T2}) give
\begin{equation*}
WC^{I,J}_{\mu,\nu}(z)=\delta^2  \frac{\varsigma(z d_1)}{\varsigma(z\delta d_1)} \frac{\varsigma(z d\delta)}{\varsigma(zd)}\frac{\varsigma(z d_2)}{\varsigma(z d_2 \delta)} H_{\delta, \nu_J+\delta}(z)H_{\mu_I^c+\delta,\nu_J^c}(z).
\end{equation*}

\begin{example}
We illustrate Theorem \ref{thm-wallcrossing} in the simplest case.  When $m=n=2$, all chambers are essentially the same: they are determined by the largest of the four parts.  We will determine the wall crossing from the chamber $\mathfrak{c}_1$ where $\mu_1$ is the largest, to $\mathfrak{c}_2$ where $\nu_1$ is the largest; these chambers border along the wall $W_{1,1}=W_{2,2}$ where $\mu_1=\nu_1$ and $\mu_2=\nu_2$.

The 2+2-point series was computed in two different ways in Examples \ref{example-221} and \ref{example-222}; Equation \ref{eq-nice22} from \ref{example-222} gives a formula for $H_{\mu,\nu}(z)$ on chamber $\mathfrak{c}_1$, and interchanging the roles of $\mu$ and $\nu$ gives a formula for $H_{\mu,\nu}(z)$ on chamber $\mathfrak{c}_2$, and so we see that we should have

\begin{equation} \label{eq-awaw}
WC^{1,1}_{\mu,\nu}(z)=\frac{1}{\mu_1\mu_2\nu_1\nu_1}\frac{\varsigma(z\mu_1\nu_1)}{\varsigma(dz)}
\left(\varsigma(\nu_1\mu_2z)\varsigma(\nu_2dz)-\varsigma(\mu_1\nu_2z)\varsigma(\mu_2dz)\right).
\end{equation}

The right hand side of Theorem \ref{thm-wallcrossing} involves $H_{\{\mu_1+\delta\}, \nu_1}(z)$ and $H_{\mu_2, \{\nu_2+\delta\}}$, which as 1+2-point series are calculable from Equation \ref{GJVformula} (Example \ref{example-GJV}).  Substituting this equation and canceling terms, we find that the right hand

\begin{equation} \label{eq-bwbw}
\delta^2  \frac{\varsigma(z \nu_1)}{\varsigma(z\delta \nu_1)} \frac{\varsigma(z d\delta)}{\varsigma(zd)}\frac{\varsigma(z \mu_2)}{\varsigma(z \mu_2 \delta)} H_{\{\delta+\mu_1\}, \nu_1}(z)H_{\mu_2,\{\nu_2+\delta\}}(z)
=\frac{\varsigma(\mu_1\nu_1z)\varsigma(\mu_2\nu_2z)\varsigma(\delta dz)}{\mu_1\mu_2\nu_1\nu_2\varsigma(dz)}.
 \end{equation}

Equating (\ref{eq-awaw}) and (\ref{eq-bwbw}), we see that in this case Theorem \ref{thm-wallcrossing} boils down to the identity
\begin{equation*}
\varsigma(\nu_1\mu_2z)\varsigma(\nu_2dz)-\varsigma(\mu_1\nu_2z)\varsigma(\mu_2dz)=\varsigma(\mu_2\nu_2z)\varsigma(\delta dz),
\end{equation*}
which, using $\nu_1=d-\nu_2, \mu_1=d-\nu_2$ and $\delta=\mu_2-\nu_2$, can be rewritten as
\begin{equation*}
\varsigma(d\mu_2z-\mu_2\nu_2z)\varsigma(\nu_2dz)+\varsigma(\mu_2\nu_2z-d\nu_2z)\varsigma(\mu_2d)+\varsigma(d\nu_2z-d\mu_2z)\varsigma(\mu_2\nu_2z)=0
\end{equation*}
which is Equation (\ref{eq-cyclicsigmaidentity}) with $\{a, b,c\}=\{\mu_2\nu_2, d\mu_2, d\nu_2\}$.

\end{example}

\subsection{Special Chambers} \label{sec-specialchambers}
In \cite{SSV}, certain chambers, called the totally negative chambers, were found where the genus zero Hurwitz numbers had a product formula.  In this section we will find a wider class of chambers where such formulas exist, and extend the result on these chambers to a product formula for the $m+n$-point series.

The method will be to find those chambers where there is only one term in the sum in Theorem \ref{thm-main}; that is, we want $CP(\mu,\nu)$ to consist of just a single element.

Note first that there is one easily described element in every $CP(\mu,\nu)$: the pattern $P$ where we take the canceling term at every possible opportunity, and never take the passing term.  We will call this the \emph{all commutator pattern}.

The all commutator pattern then gives us a natural total ordering $\varphi$ of the $m+n$ parts of $\mu$ and $\nu$.  To make this compatible with the usual orderings on $\mu$ and $\nu$, we will label the parts of $\mu$ and $\nu$ out from the center; that is, the vacuum expectation we are computing is
\begin{equation*}
\left\langle \EE{m}{\emptyset}\cdots \EE{1}{\emptyset}\EE{\emptyset}{1}\cdots\EE{\emptyset}{n}\right\rangle.
\end{equation*}

The ordering $\varphi$ is defined as follows: the larger of $\mu_1$ and $\nu_1$ is the first element $\varphi(1)$, and the smaller of them is the second, $\varphi(2)$.  The remaining terms are ordered occurring to when they are first involved in a commutator in the all commutator pattern.

Let $\varphi(k)$ denote the $k$th term in this total ordering.  Either $\varphi(k)$ is a part of $\mu$ or a part of $\nu$.  We will say $k\sim\ell$ if $\varphi(k)$ and $\varphi(\ell)$ belong the same partition, and $k\nsim \ell$ if they do not.  We will frequently use $k\lnsim \ell$ to mean $k<\ell$ and $k\nsim\ell$.

\begin{definition} \label{def-totneg}
The \emph{$\varphi$-totally negative chamber} is the chamber defined by the following inequalities, for all $k>1$:
\begin{equation*}
\varphi(k)>\sum_{\ell\gnsim k} \varphi(\ell).
\end{equation*}
\end{definition}

\begin{remark}
It is not immediately obvious that the inequalities in Definition \ref{def-totneg} define a single chamber, as they do not contain all walls $W_{I,J}$.  However, they imply an inequality for each $I,J$, as we now explain.

We may assume $1\notin I$, as otherwise we can work with $I^c$ and $J^c$ instead.  Let $k\in [m+n]$ be minimal such that $\varphi(k)$ is one of the parts in $\mu_I$ or $\nu_J$; say $\varphi(k)\in \mu_I$.  Then every part $\nu_j$ of $\nu_J$ is of the form $\varphi(\ell)$ with $\ell>k$, and so
\begin{equation*}
\sum_{i\in I} \mu_i \geq \varphi(k) >\sum_{\ell\gnsim k} \varphi(\ell) \geq \sum_{j\in J} \nu_j.
\end{equation*}
\end{remark}

\begin{lemma}
The set $CP(\mu,\nu)$ consists of a single element exactly if $\mu,\nu$ are in a totally $\varphi$-negative chamber.
\end{lemma}

\begin{proof}
Suppose $\mu,\nu$ are in a totally $\varphi$-negative chamber, and suppose there was a commutation pattern in $CP(\mu,\nu)$ that contained a passing term.  Suppose the first $k-1\geq 0$ interactions are commutators, so that the first passing term involves $\varphi(k)$  passing an $\E$ term containing exactly the $\varphi(j)$ with $j<k$.  We can assume $\varphi(k)\in \mu$, so that immediately after we take the passing term the far right of the vacuum expectation will be
\begin{equation*}
\E_{\varphi(k)}(0) \prod_{j\gnsim k} \E_{-\varphi(j)}(z\varphi(j))\left |0\right\rangle
\end{equation*}
This product of $\E$ terms has energy
\begin{equation*}
\varphi(k)-\sum_{j\gnsim k} \varphi(j)>0,
\end{equation*} and so annihilates the vacuum, and hence this passing term is zero.

If, on the other hand, $\mu,\nu$ were in a chamber with
\begin{equation*}
\varphi(k)<\sum_{j\gnsim k} \varphi(j),
\end{equation*} then we can easily construct another nonzero passing term: always take the canceling term except for the one passing term described above; that is, pass the $\E$ with $\phi(k)$ through the $\E$ with $\phi(j)$ for $j<k$.
\end{proof}

\subsubsection{The product formula}
In a totally $\varphi$-negative chamber $\mathfrak{c}$, Theorem $\ref{thm-main}$ then has a particularly attractive form.

Since every commutator taken involves an $\EE{I}{J}$ with one of $I, J$ empty, every matrix appearing in a $\varsigma$ will be upper or lower triangular, and so we see that for $P$ the all commutator term, we have
\begin{equation*}
\bvs{I_\ell^P}{J^P_\ell}{K_\ell^P}{L_\ell^P}=\varsigma\left(z\cdot \varphi(\ell) \cdot \sum_{j\lnsim \ell} \varphi(j) \right).
\end{equation*}

\bibliographystyle{alpha}
\bibliography{hurwitzbib}

\end{document}